\documentclass[11pt]{amsart}
\usepackage{amsmath,amssymb,amsthm}
\usepackage{amsmath,amssymb,amsthm,amscd}
\usepackage[frame,cmtip,arrow,matrix,line,graph,curve]{xy}
\usepackage{graphpap, color}
\usepackage[mathscr]{eucal}
\usepackage{color}
\numberwithin{equation}{section}
 \numberwithin{equation}{section}

\renewcommand{\oddsidemargin}{5mm}

\newtheorem{prop}{Proposition}
\newtheorem{thm}[prop]{Theorem}
\newtheorem{lemma}[prop]{Lemma}
\newtheorem{cor}[prop]{Corollary}

\newtheorem{remark}[prop]{Remark}
\newtheorem{define}[prop]{Definition}

\theoremstyle{definition}

\theoremstyle{remark}

\newcommand{\p}{\partial}

\def\R{\mathbb R}

\def\S{\mathbb S}

\def\f{\frac}

\def\la{\lambda}

\title[On the self-shrinking systems]{On the self-shrinking systems in arbitrary
codimensional spaces}
\author{Qi Ding}
\address{Institute of Mathematics\\ Fudan University \\ Shanghai
200433, China} \email{09110180013@fudan.edu.cn}
\author{Zhizhang Wang}
\address{Institute of Mathematics\\ Fudan University \\ Shanghai
200433, China} \email{youxiang163wang@163.com}
\begin{document}
\maketitle
\begin{abstract}
In this paper, we discuss the self-shrinking systems in higher
codimensional spaces. We mainly obtain several Bernstein type
results and a sharp growth estimate.

\end{abstract}
\section{Introduction}
Let $M$ be an $n$-dimensional manifold immersed into the Euclidean
space $\R^{n+m}$. Denote the immersed map by $X : M \rightarrow
\R^{n+m}$. Let $(\cdots)^N$ be the projective map from the trivial
bundle $\R^{n+m}\times M$ onto the normal bundle over $M$. Then we can
define the mean curvature vector of the immersed manifold $M$ into
$\R^{n+m}$, seeing \cite{Xin},
$$H\ \ = \ \ \sum_{i=1}^n\overline{\nabla}^N_{e_i}e_i.$$ Here,
$\{e_i\}_{i=1}^n$ is a local orthonormal tangent frame of $M$,
$\overline{\nabla}$ is the canonical connection of $\R^{n+m}$ and $\overline{\nabla}^N$ is the connection of normal bundle $NM$. If we
let the position vector $X$ move in the direction of the mean
curvature vector $H$, then it gives the \emph{mean curvature flow},
namely,
\begin{eqnarray}\label{1.1}
\f{\p X}{\p t}&=&H,\qquad \mathrm{on}\
M\times[0,T).
\end{eqnarray}
Here, $[0,T)$ is the maximal finite time interval on which the flow
exists.

The immersed manifold $M$ is said to be a \emph{self-shrinker} (see
 \cite{H2} or \cite{Smo} for details), if it satisfies a quasi-linear elliptic
system,
\begin{eqnarray}\label{1.2}
H \ \  = \ \ -X^N.
\end{eqnarray}
Self-shrinkers are an important class of solutions to (\ref{1.1}).
Not only are they shrinking homothetic under mean curvature flow
(see \cite{CM} for detail), but also they describe all possible blow
ups at a given singularity of a mean curvature flow.

Now, we give a roughly brief review about the self-shrinkers. For
curves case, U. Abresch and J. Langer \cite{AL} gave a complete
classification of all solutions to (\ref{1.2}). These curves are now
called  Abresch-Langer curves.

For codimension 1 case, K. Ecker and G. Huisken \cite{EH} showed
that a self-shrinker is a hyperplane, if it is an entire graph with
polynomial volume growth. Let $\vec{n}$ be the unit outward normal
vector of $M^n$ in $\R^{n+1}$ and $|B|$ be the norm of the
second fundamental form of $M$. In \cite{H2} and \cite{H3}, G. Huisken
gave the classification theorem that the only possible smooth
self-shrinkers in $\R^{n+1}$ satisfying mean curvature of $M$, $\mathrm{div}(\vec{n})\geqq  0$,  $|B|$  bounded
and polynomial volume growth are isometric to $\Gamma\times
\R^{n-1}$ or $\S^k\times\R^{n-k}$ ($0\leqq
 k\leqq
n$). Here, $\Gamma$ is an Abresch-Langer curve and $\S^k$ is a
$k$-dimensional sphere. In general, the classification of
self-shrinkers seems difficult. However T.H. Colding and W.P.
Minicozzi II \cite{CM1} offer a possibility. Recently, They
\cite{CM}, \cite{CM7} showed that a long-standing conjecture of
Huisken is right. The conjecture is classifying the singularities of
mean curvature flow starting at a generic closed embedded surface in
$\R^3$. They also proved that G. Huisken's classification theorem
still holds without the assumption that $|B|$ is bounded. Meanwhile,
L. Wang \cite{Lwang} proved that an entirely graphic self-shrinker
should be a hyperplane without any other assumption.

Another special case is that the self-shrinker is a Lagrangian
graph. In \cite{ACH2}, A. Chau, J.Y. Chen, and W.Y. He firstly study
Lagrangian self-shrinkers in Euclidean space.  Recently, R. Huang
and Z. Wang \cite{HW} obtained the following two results for
Lagrangian graphs. In pseudo-Euclidean space, if the Hessian of the
potential function has order two decay in the infinity, the
self-shrinker defined by the potential is a linear subspace. In
Euclidean space, if the potential function is convex or concave, the
corresponding self-shrinker is a linear subspace. After that, A.
Chau, J.Y. Chen and Y. Yuan \cite{CCY} improved and generalized
previous results. In Euclidean space, they drop the assumption that
the potential should be convex or concave. In pseudo-Euclidean,
using different method, they proved a similar result. They also
generalized their pseudo-Euclidean result to complex case, which
relates to a special class of self-shrinking K\"ahler-Ricci
solitons.

In arbitrary codimension case, K. Smozyk in \cite{Smo} proved two
results. Suppose the manifold $M^m$ is a compact self-shrinker, then
$M$ is spherical if and only if $H\neq 0$ and its principal normal
vector $\nu=H/|H|$ is parallel in the normal bundle. Suppose $M^m$
is a complete connected self-shrinker with $H\neq 0$, parallel
principal normal and having uniformly bounded geometry, then $M$
must be $\Gamma\times \R^{m-1}$ or $\tilde{M}^r\times \R^{m-r}$.
Here $\Gamma$ is an Abresch-Langer curve and $\tilde{M}$ is a
minimal submanifold in sphere.

The study of higher codimensional self-shrinkers seems difficult for
us. Hence, in the present paper, we only consider the simplest case
that $M$ is a smooth graph.

Assume $M$ to be a codimension $m$ smooth graph in
$\mathbb{R}^{n+m}$,
$$M=\{(x_1,\cdots,x_n,u^1,\cdots,u^m)\in \mathbb{R}^{n+m}; x_i\in \mathbb{R}, u^{\alpha}=u^{\alpha}(x_1,\cdots,x_n)\},$$
where $i=1,\cdots,n$ and $\alpha=1,\cdots,m$. Denote
\begin{eqnarray}\label{x,u}
x=(x_1,x_2,\cdots,x_n)\in \mathbb{R}^n;\ \ u=(u^1,\cdots,u^m)\in
\mathbb{R}^m.
\end{eqnarray}
In Euclidean and pseudo-Euclidean space, denote the metric on $M$ be
$g=\sum_{i,j=1}^ng_{ij}dx_idx_j$, and
$$g_{ij}\ \ = \ \ \delta_{ij}\pm\sum_{\alpha=1}^mu^{\alpha}_iu^{\alpha}_j.$$
Here, $u^\alpha_i=\dfrac{\p u^\alpha}{\p x_i}$, $''+''$ is chosen in
Euclidean space and $''-''$ is chosen by space-like submanifolds in
pseudo-Euclidean space with index $m$, which means the background
metric in $\R^{n+m}$ is
$$ds^2=\sum_{i=1}^ndx^2_i-\sum_{\alpha=1}^mdx_{n+\alpha}^2,$$ seeing \cite{Xin} for details. Then,
\eqref{1.2} is equivalent to the following elliptic system
\begin{eqnarray}\label{meq}
\sum_{i,j=1}^ng^{ij}u^{\alpha}_{ij}&=&-u^{\alpha}+x\cdot D u^{\alpha}.
\end{eqnarray}
Here, $Du^{\alpha}=(u^{\alpha}_1,\cdots, u^{\alpha}_n)$, and
$''\cdot''$ denotes the canonical inner product in $\R^n$. The
details of the calculation will appear in section 2.

In section 3, a special class of functions is important. It is,
\begin{define} Assume $\phi$ to be a $C^2$-function defined on
$\mathbb{R}^n$. We call that $\phi$ is a strong-sub-harmonic
(shortly, SSH) function, if $\phi$ satisfies
\begin{eqnarray}
\sum_{i,j=1}^ng^{ij}\phi_{ij}-x\cdot D \phi&\geqq& \varepsilon
\sum_{i,j=1}^ng^{ij}\phi_i\phi_j,
\end{eqnarray}
for some small positive constant $\varepsilon$.
\end{define}
Based on the Lemma \ref{sta}, we know that the only possible SSH
functions are constants in $\mathbb{R}^n$. Hence, our crucial
problem becomes to find appropriate SSH-functions. We prove that the
volume element function is  SSH, if one of the following three
assumptions is satisfied. The first one is that the multiple of two
different eigenvalues of the singular value of $du$ (c.f.
\cite{Wang1}) is no more than one. The second one is that the volume
element function is less than some positive constant $\beta<9$. The
last one is  a communication formula which includes the  normal
bundle flat case and codimension 1 case. Then, the volume element
function is a constant which implies Bernstein type results. The
corresponding minimal submanifolds cases have appeared in
\cite{JXY}, \cite{SWX} and \cite{Wang1}.  At the end of this
section,  we prove that there is no non-trivial rotationary
symmetric self-shrinking graph, but for submanifolds, we have
non-trivial examples.  S. Kleene and N.M. M{\o}ller \cite{KM} gave
the classification for complete embedded revolutionary
hypersurfaces.

In section 4, we give the sharp growth estimate of $u^{\alpha}$.
Obviously, the linear functions satisfy \eqref{meq}. In Euclidean
space, we have the following estimate.
\begin{thm}\label{thm1}
The solution $u^{\alpha}$  of system \eqref{meq} defining a graph in
Euclidean space is linear growth. In fact,  we have the following
estimate,
\begin{eqnarray}\label{1.3}
|u(x)|^2&\leqq& (\frac{2|x|^2}{3n}+1)(\sup_{|x|\leqq
2\sqrt{3n}}|u(x)|^2+12n),
\end{eqnarray}
where $|u(x)|^2=\sum_{\alpha=1}^m(u^\alpha(x))^2$.
\end{thm}
Then, we generalize the above estimate to the following linear
elliptic system,
\begin{eqnarray}\label {gel}
\sum_{i,j=1}^na^{ij}u^{\alpha}_{ij}&=&-u^{\alpha}+x\cdot D u^{\alpha},
\end{eqnarray}
with two extra assumptions. One is that  the coefficient matrix
$(a^{ij})$ have a uniform upper bound. Namely, there is some
positive constant $\sigma$, such that
\begin{eqnarray}\label{gelcond}
\sum_{i,j=1}^na^{ij}\xi_i\xi_j&\leqq& \sigma |\xi|^2\end{eqnarray}
holds for any vector $\xi\in \R^n$. The other is that there be three
positive constants $r_0$, $c$ and $\tau$, such that for $|x|\geqq
r_0$, we have
\begin{eqnarray}\label{gelcond2}
\sum_{i,j=1}^n\sum^m_{\alpha=1}a^{ij}(x)u^{\alpha}_i(x)u^{\alpha}_j(x)
&\leqq& c |x|^{2\tau-2}.
\end{eqnarray}
Roughly speaking, we prove that any solution $u^{\alpha}$ satisfying
\eqref{gel}, \eqref{gelcond} and \eqref{gelcond2} has polynomial
growth. See Theorem \ref{thm15} for details.

In the last part, we discuss the Bernstein type theorems for the
space-like self-shrinking graph in pseudo-Euclidean space with index
$m$. Firstly, We prove that, if the metric has a positive lower
bound, the graph should be a linear subspace. If a self-shrinking
graph also pass through the origin, we have the following better
result.
\begin{thm}\label{thm2}
Assume $u^{\alpha}$ to be a solution of the system \eqref{meq},
defining a graph in pseudo-Euclidean space with index $m$. Further,
assume that $u^{\alpha}(0)=0$ and that $\det(g)$ has subexponential
decay in the sense that
$$\lim_{|x|\rightarrow\infty}\f{\log\det(g_{ij}(x))}{|x|}=0.$$
Then $M$ is a linear
subspace.
\end{thm}

This paper is organized as follows. In the second section, we
calculate explicitly the self-shrinking system in higher
codimensional spaces. In the third section, we derive several
Bernstein type results in Euclidean space. In the forth section, we
obtain the sharp growth estimate for self-shrinking function
$u^{\alpha}$ in Euclidean space. Then, we generalize it to some
class of  linear elliptic systems. In the last section, we obtain
two results about space-like self-shrinking graph in
pseudo-Euclidean space with index $m$.

 \vspace{5mm}
 {\bf Acknowledgment:}
The authors wish to express their sincere gratitude to Professor
Y.L. Xin and Professor J.X. Fu for their valuable suggestions and
comments.

\section{The self-shrinking systems}
From now on, Einstein convention of summation over repeated indices
will be adopted. We also assume that Latin indices $i,j,\cdots$,
 Greek indices $\alpha,\beta,\cdots$
 and capital Latin indices $A,B,\cdots$ take
 values in the sets $\{1,\cdots,m\}$, $\{m+1,\cdots,m+n\}$ and
 $\{1,\cdots,m+n\}$ respectively.

Firstly, we calculate the system of the self-shrinkers in higher
codimensional spaces. For the graph $M$, denote
$$X=(x_1,\cdots,x_n,u^1,\cdots,u^m)=(x,u).$$
Let $\{E_A\}_{A=1}^{n+m}$ be the canonical orthonormal basis of
$\R^{n+m}$. Namely,  every component of the vector $E_A$ is $0$,
except that the $A$-th component is $1$. Then
$$e_i=E_i+\sum_{\alpha} u^\alpha_iE_{n+\alpha},\ \ \text{for} \ \  i\in\{1,\cdots,n\}$$ give a tangent frame on $M$.

In Euclidean space, the metric on $M$ is
$$g_{ij}=\langle e_i,e_j\rangle=\delta_{ij}+\sum_{\alpha}u^{\alpha}_iu^{\alpha}_j.$$
Here, $\langle\cdot,\cdot\rangle$ is the canonical inner product in
$\R^{n+m}$. Then there are $m$ linear independent unit normal vectors,
$$n_{\alpha}\ \ =\ \ \frac{1}{(1+|D u^{\alpha}|^2)^{1/2}}(-\sum_{i}u^{\alpha}_iE_i+E_{n+\alpha}),\ \mathrm{for} \ \alpha\in\{1,\cdots,m\}.$$
Note that $\{n_{\alpha}\}$ are not necessarily  orthogonal with each
other. We always denote $D u^{\alpha}=\sum_{i}u^{\alpha}_iE_i$.

Now we define
the $\alpha$-th mean curvature component is
\begin{eqnarray}
H^{\alpha}=\langle H, n_{\alpha}\rangle.
\end{eqnarray}
Then, we have
\begin{eqnarray}
\langle H,
n_{\alpha}\rangle&=&g^{ij}\langle\overline{\nabla}_{e_i}e_j,
n_{\alpha}\rangle
\ \ = \ \ g^{ij}\langle u^\beta_{ij}E_{n+\beta}, n_{\alpha}\rangle\nonumber\\
&=&\dfrac{g^{ij}u^{\alpha}_{ij}}{(1+|D
u^{\alpha}|^2)^{1/2}}.\nonumber
\end{eqnarray}
Then by \eqref{1.2},
\begin{eqnarray}\label{2.2}
\frac{g^{ij}u^{\alpha}_{ij}}{(1+|D
u^{\alpha}|^2)^{1/2}}=H^{\alpha}=-\langle X,
n_{\alpha}\rangle=\frac{-u^{\alpha}+x\cdot D u^{\alpha}}{(1+|D
u^{\alpha}|^2)^{1/2}}.
\end{eqnarray}
We obtain \eqref{meq}.

In pseudo-Euclidean space with index $m$, for the space-like
graphes, the metric and the normal directions are
\begin{eqnarray}\label{pmtc}
g_{ij}=\delta_{ij}-\sum_{\alpha}u^{\alpha}_iu^{\alpha}_j,\ \ \text{
and }\ \ n_{\alpha}=\frac{1}{(1-|D
u^{\alpha}|^2)^{1/2}}(\sum_{\alpha}u^{\alpha}_iE_i+E_{n+\alpha}).
\end{eqnarray}
Then we obtain a similar system as \eqref{meq}, where we replace the
metric $g_{ij}$ by \eqref{pmtc}.
\begin{remark}
For a Lagrangian graph, $m=n$, and there is some potential function
$v$ (c.f. \cite{HL}) , such that  $$u^{\alpha}= \frac{\partial
v}{\partial x_{\alpha}}.$$ Then inserting it into \eqref{meq}, and
integrating, we obtain the equations of Lagrangian self-shrinkers in
Euclidean and pseudo-Euclidean space. They are
\begin{eqnarray}\label{2.4}
tr\arctan D^2v=-2v+x\cdot D v,
\end{eqnarray}
or
\begin{eqnarray}\label{2.5}
\frac{1}{2}tr\ln \frac{I+D^2v}{I-D^2v}=-2v+x\cdot D v,
\end{eqnarray}
respectively. Here, $tr$ means taking the trace of  matrices.
\end{remark}
For the first equation, we let $$w(y)=2v(\frac{1}{\sqrt{2}}y).$$
Then \eqref{2.4} becomes the  equation (2) in \cite{ACH2}. For the
second equation, first, we let
$$\eta(y)=\frac{4}{n}v(\frac{\sqrt{n}}{2}y).$$ The equation becomes
$$tr\ln \frac{I+D^2\eta}{I-D^2\eta}=n(-\eta+\frac{1}{2}y\cdot D \eta) .$$
 Then, using Lewy rotation \cite{Y},
\begin{equation}
\left\{ \begin{aligned}\bar{x}&=\frac{x-D \eta(x)}{\sqrt{2}} \\
D w(\bar{x})&=\frac{x+D \eta(x)}{\sqrt{2}}
\end{aligned} \right.,\nonumber
\end{equation}
we get
$$\ln \det D^2 w=tr\ln D^2 w =n(-w+\frac{1}{2}\bar{x}\cdot D
w).$$ In the last equality, we use a similar trick appearing in
\cite{HW}. Then we obtain the self-shrinking equations of Lagrangian
mean curvature flow in pseudo-Euclidean space. This equation appears
 in \cite{HW} and \cite{CCY} at first.

\section{Some Bernstein type results in Euclidean space}
For any function $\phi$ on $\mathbb{R}^n$, denote an operator
\begin{eqnarray}
L_a\phi&=&a^{ij}\phi_{ij}-x\cdot D \phi.
\end{eqnarray}
Here, $(a^{ij})$ is the inverse  of a positive definite matrix
$(a_{ij})$ at every point in $\R^n$. The critical point to obtain
Bernstein type results is the following lemma. It seems that the
second inequality of the following lemma similar to some stability
condition. The proof of the lemma slightly modifies from the last
part of \cite{HW}. But for the readers' convenient, we include it
here.
\begin{lemma}\label{sta}
Let the minimum eigenvalue of the matrix $(a_{ij})$ be $\nu(x)$.
Assume
\begin{eqnarray}\label{lemmacond}
\liminf_{|x|\rightarrow +\infty}\nu(x)|x|^2>n.
\end{eqnarray}
For any smooth function $\phi$, if there is a small positive
constant $\varepsilon$, such that,
$$L_a\phi\ \ \geqq \ \ \varepsilon a^{ij}\phi_i\phi_j,$$
then $\phi$ is a constant.
\end{lemma}
\begin{proof}
For $0<k<1$, let
\begin{eqnarray*}\label{e4.1}
\eta(x)=\left\{\begin{matrix}1& |x|\leqq R_0\\
-k(|x|^2-R_0^2)+1 & |x|\geqq R_0\end{matrix}\right..
\end{eqnarray*}
Here $R_0$ is a constant which will be determined later. Then the
function $\eta e^{C\phi}$ achieves its maximum in the bounded set
$$\{x\in \mathbb{R}^n;\eta>0\},$$ where $C$ is a positive constant which also will be determined later. If the maximum point $p$ is in
the set $\{x\in \R^n;|x|>R_0\}$, then at $p$, we have
\begin{eqnarray}\label{3.3}
\eta_i+C\eta\phi_i=0.
\end{eqnarray}
At $p$, using \eqref{3.3}, we have,
\begin{eqnarray}\label{compute}
e^{-C\phi}a^{ij}(\eta
e^{C\phi})_{ij}&=&a^{ij}\eta_{ij}+2Ca^{ij}\eta_i\phi_j+C\eta
a^{ij}\phi_{ij}+C^2\eta a^{ij}\phi_{i}\phi_j\\
&\geqq&-2k\sum_ia^{ii}-C^2\eta a^{ij}\phi_{i}\phi_j+C\eta
(x\cdot D\phi+\varepsilon a^{ij}\phi_i\phi_j)\nonumber\\
&=&-x\cdot D\eta-2k\sum_ia^{ii}+(C\varepsilon-C^2)\eta
a^{ij}\phi_{i}\phi_j\nonumber\\
&=&2k(|x|^2-\sum_ia^{ii})+(C\varepsilon-C^2)\eta
a^{ij}\phi_{i}\phi_j\nonumber.
\end{eqnarray}
By condition \eqref{lemmacond}, there is a sufficient large radius
$R_1$, such that in $\mathbb{R}^n \setminus B_{R_1}$, we have
$$a^{ii}<\frac{|x|^2}{n},$$ for any $i$, where $B_{R_1}$ denotes an open ball of radius $R_1$ centred at origin.
Using \eqref{compute} and taking the constant $C$ sufficient small,
we obtain a contradiction if $p\in (\mathbb{R}^n\setminus
B_{R_1})\bigcap \{x\in\R^n;|x|>R_0\}$.

Assume that the function $\phi$ is not a constant in $\mathbb{R}^n$.
Then there is a ball $B_{R_0}$ with radius $R_0\geqq R_1$, such that
the function $\phi$ is not a constant in $B_{R_0}$. Suppose that
$\phi$ achieves its maximum value in $B_{R_0}$. Since $L_a\phi\geqq
0$, applying strong maximum principle, we obtain $\phi$ is a
constant, which is a contradiction. Hence, $\phi$ achieves its
maximum value only on the boundary $\partial B_{R_0}$. Similarly, in
$B_{\sqrt{R_0^2+1}}$, $\phi$ also achieves its maximum value only on
the boundary $\partial B_{\sqrt{R_0^2+1}}$. We assume that the
points $p_1$ and $p_2$ are maximum value points with respect to
$\partial B_{R_0}$ and $\partial B_{\sqrt{R_0^2+1}}$, namely,
\begin{eqnarray}
\max_{\overline{B}_{R_0}}\phi=\phi (p_1),\ \
\max_{\overline{B}_{\sqrt{R_0^2+1}}}\phi=\phi (p_2).\nonumber
\end{eqnarray}
Then
\begin{eqnarray}
\phi (p_1)&\leqq&\phi (p_2).\nonumber
\end{eqnarray}
But the equality is not valid. In fact, if the equality holds, then
the function $\phi$ achieves  its maximum value in the interior of
the domain $B_{\sqrt{R_0^2+1}}$, which is a contradiction. Thus, we
can choose $k$ sufficiently small, such that
\begin{eqnarray}
(\eta e^{C\phi})(p_1) = (e^{C\phi})(p_1)<((1-k)e^{C\phi})(p_2)=(\eta
e^{C\phi})(p_2).\nonumber
\end{eqnarray}
This means that, for fixed $\phi$, we can choose suitable $k$, such
that the maximum value of $\eta e^{C\phi}$ only occurs in the set
$$\{x\in\mathbb{R}^n||x|>R_0\geqq R_1\}.$$ But we have proved that it is
impossible. Thus, all above discussions imply the function $\phi$
should be a constant.
\end{proof}
\begin{remark}
If the matrix $(a_{ij})$ is the induced metric of a graph in
Euclidean space, we can drop the condition \eqref{lemmacond}. In
fact, Lemma \ref{sta} says that any SSH-function should be a
constant in Euclidean space.
\end{remark}
In \cite{CCY}, A. Chau, J.Y. Chen and Y. Yuan proved the following
theorem. Here, we give another proof.
\begin{thm}
For a Lagrangian graph, every self-shrinker should be a linear
subspace.
\end{thm}
\begin{proof}
By \cite{CCY}, the phrase function $\Theta$ of a Lagrangian graph
satisfies
$$g^{ij}\Theta_{ij}-x\cdot D \Theta=0.$$ Then
\begin{eqnarray}
g^{ij}(e^{\Theta})_{ij}-x\cdot D (e^{\Theta})=e^{-\Theta}
g^{ij}(e^{\Theta})_i(e^{\Theta})_j.\nonumber
\end{eqnarray}
Since $\Theta$ is bounded, $e^{\Theta}$ is SSH. By Lemma \ref{sta},
$\Theta$ is a constant. Then, by the same argument of \cite{CCY}, we
obtain the result.
\end{proof}
A simple result is the following.
\begin{prop}\label{4}
If $f$ is a nonnegative function satisfying
\begin{eqnarray}\label{fequ}
g^{ij}f_{ij}=-f+x\cdot D f,\end{eqnarray} then $f\equiv0$.
\end{prop}
\begin{proof}
Since $f\geqq 0$, we consider the function
$$\phi=e^{-f}.$$ Then, we have
\begin{eqnarray}
g^{ij}\phi_{ij}&=&e^{-f}g^{ij}f_if_j-e^{-f}g^{ij}f_{ij}\nonumber\\
 &=&e^{-f}g^{ij}f_if_j+e^{-f}f-e^{-f}x\cdot D f\nonumber.
\end{eqnarray}
It is
\begin{eqnarray}
g^{ij}\phi_{ij}-x\cdot D
\phi&=&e^{-f}g^{ij}f_if_j+e^{-f}f\nonumber\\
&\geqq& e^{-2f}e^{f}g^{ij}f_if_j\nonumber\\
&\geqq&g^{ij}\phi_i\phi_j\nonumber.
\end{eqnarray}
So $\phi$ is SSH. By Lemma \ref{sta}, we obtain that  $e^{-f}$ is a
constant. Namely, $f$ is a constant. But \eqref{fequ} tells us that
the constant must be $0$.
\end{proof}

\begin{cor}
For the system \eqref{meq}, assume that every $u^{\alpha}$ is only
in one side of some hyperplane. Namely, there are constant vectors
$b^{\alpha}$, such that
$$u^{\alpha}\geqq b^{\alpha}\cdot x,\ \ \text{for}\ \  \alpha=1,\cdots, m.$$ Then $M$ is a linear subspace.
\end{cor}
\begin{proof}
In Proposition \ref{4}, let $f=u^{\alpha}-b^{\alpha}\cdot x$. Then
$f=0$, this means that $u^{\alpha}$ is a linear function and $M$ is
a linear subspace.
\end{proof}

A possible SSH-function is the volume element function . See
\cite{Wang1} and \cite{Y}. In Euclidean space, we let
\begin{eqnarray}\label{3.6}
\phi=\ln \det (g_{ij}).
\end{eqnarray}
Then
\begin{eqnarray}\label{3.1}
g^{ij}\phi_{ij}
&=&-2\sum_{\alpha,\beta}g^{ij}g^{pk}g^{ql}u^{\alpha}_{pi}u^{\alpha}_qu^{\beta}_{kj}u^{\beta}_l
-2\sum_{\alpha,\beta}g^{ij}g^{pk}g^{ql}u^{\alpha}_{qi}u^{\alpha}_pu^{\beta}_{kj}u^{\beta}_l\\
&&+2\sum_{\alpha}g^{ij}g^{pq}u^{\alpha}_{pi}u^{\alpha}_{qj}+2\sum_{\alpha}g^{ij}g^{pq}u^{\alpha}_{pij}u^{\alpha}_{q}.\nonumber
\end{eqnarray}
By system \eqref{meq}, we have
\begin{eqnarray}
g^{ij}u^{\alpha}_{ijp}&=&x\cdot D
u^{\alpha}_p+g^{ik}g^{jl}g_{klp}u^{\alpha}_{ij}\nonumber\\
&=&x\cdot D
u^{\alpha}_p+2\sum_{\beta}g^{ik}g^{jl}u^{\beta}_{kp}u^{\beta}_lu^{\alpha}_{ij}\nonumber.
\end{eqnarray}
Then by \eqref{3.1}, we get
\begin{eqnarray}
g^{ij}\phi_{ij}&=&-2\sum_{\alpha,\beta}g^{ij}g^{pk}g^{ql}u^{\alpha}_{pi}u^{\alpha}_qu^{\beta}_{kj}u^{\beta}_l
+2\sum_{\alpha,\beta}g^{ij}g^{pk}g^{ql}u^{\alpha}_{qi}u^{\alpha}_pu^{\beta}_{kj}u^{\beta}_l\nonumber\\
&&+2\sum_{\alpha}g^{ij}g^{pq}u^{\alpha}_{pi}u^{\alpha}_{qj}
+2\sum_{\alpha}g^{pq}u^{\alpha}_{q}x\cdot D u^{\alpha}_p\nonumber.
\end{eqnarray}
It is
\begin{eqnarray}
L_g(\phi)&=&-2\sum_{\alpha,\beta}g^{ij}g^{pk}g^{ql}u^{\alpha}_{pi}u^{\alpha}_qu^{\beta}_{kj}u^{\beta}_l
+2\sum_{\alpha,\beta}g^{ij}g^{pk}g^{ql}u^{\alpha}_{qi}u^{\alpha}_pu^{\beta}_{kj}u^{\beta}_l\nonumber\\
&&+2\sum_{\alpha}g^{ij}g^{pq}u^{\alpha}_{pi}u^{\alpha}_{qj}.\nonumber
\end{eqnarray}
By simply calculation, the operator $L_g$ is invariant under the
orthonormal transformations from $\R^n$ to $\R^n$ and from $\R^m$ to
$\R^m$. Then,  at any fixed point, we can choose a coordinate system
$\{x_1,\cdots,x_n\}$ on $\R^n$ and $\{u^1,\cdots,u^m\}$ on $\R^m$,
such that
\begin{equation}\label{diag}
\frac{\partial u^\alpha}{\partial x_i}=\lambda_i\delta^\alpha_i,
\end{equation}
where $\lambda_i\geqq 0$ are the singular values of $du$, and
$\lambda_i=0$ for $i\in\{\min\{m,n\}+1,\cdots,\max\{m,n\}\}$. If
$n>m$, we let $u^A_{ij}=0$ for convenience, where
$A\in\{m+1,\cdots,n\}$. Then, we get
\begin{eqnarray}\label{3.8}
&L_g(\phi)=&-2\sum_{i,p,q}\frac{\la_q^2(u^{q}_{pi})^2}{(1+\lambda^2_i)(1+\lambda^2_p)(1+\lambda^2_q)}
+2\sum_{\alpha,i,p}\frac{(u^{\alpha}_{pi})^2}{(1+\lambda^2_i)(1+\lambda^2_p)}\\
&&+2\sum_{i,p,q}\frac{\la_p\la_qu^p_{qi}u^q_{pi}}{(1+\lambda^2_i)(1+\lambda^2_p)(1+\lambda^2_q)}.
\nonumber
\end{eqnarray}

Now we have the following Theorem.
\begin{thm}\label{10}
For the system \eqref{meq},  if one of the following three
assumptions holds:

\noindent (i) $\lambda_i\lambda_j\leqq 1$ for any $i\neq j$;

\noindent (ii) $\det(g_{ij})\leqq \beta< 9$, where $\beta$ is a
positive constant;

\noindent (iii)
$u^{\alpha}_{qi}g^{ij}u^{\beta}_{jp}=u^{\beta}_{qi}g^{ij}u^{\alpha}_{jp}$
for any $\alpha, \beta$;

\noindent then $M$ is a linear subspace.
\end{thm}
\begin{proof}

\noindent (i) Since
$$2\la_p\la_qu^p_{qi}u^q_{pi}\leqq2|u^p_{qi}u^q_{pi}|\leqq (u^p_{qi})^2+(u^q_{pi})^2,$$
for $p\neq q$, then by \eqref{3.8}, we get
\begin{eqnarray}
L_g\phi&\geqq&
-2\sum_{i,p,q}\frac{\la_q^2(u^{q}_{pi})^2}{(1+\lambda^2_i)(1+\lambda^2_p)(1+\lambda^2_q)}
+2\sum_{i,p,q}\frac{(u^q_{pi})^2(1+\lambda^2_q)}{(1+\lambda^2_i)(1+\lambda^2_p)(1+\lambda^2_q)}\\
&&-2\sum_{i,p\neq
q}\frac{(u^q_{pi})^2}{(1+\lambda^2_i)(1+\lambda^2_p)(1+\lambda^2_q)}
+2\sum_{i,p}\frac{\la_p^2(u^p_{pi})^2}{(1+\lambda^2_i)(1+\lambda^2_p)^2}
\nonumber\\
&=&2\sum_{i,p}\frac{(u^p_{pi})^2}{(1+\lambda^2_i)(1+\lambda^2_p)^2}
+2\sum_{i,p}\frac{\la_p^2(u^p_{pi})^2}{(1+\lambda^2_i)(1+\lambda^2_p)^2}\nonumber\\
&\geqq&\sum_i\frac{2}{n(1+\lambda^2_i)}\left(\sum_{p}\frac{\la_pu^p_{pi}}{1+\lambda^2_p}\right)^2
\ \ =\ \ \frac{1}{2n}g^{ij}\phi_i\phi_j,\nonumber
\end{eqnarray}
where we have used the Cauchy inequality in the second inequality.
By Lemma \ref{sta}, $\phi$ is a constant. Then the above inequality
implies $ u_{pi}^{\alpha}=0$. Then $u^{\alpha}$ are linear
functions.

\noindent (ii) By \eqref{3.8}, we have
\begin{eqnarray}
L_g(e^{\phi/2})&=&e^{\phi/2}\frac{1}{2}L_g(\phi)+e^{\phi/2}\frac{1}{4}g^{ij}\phi_i\phi_j\nonumber\\
&=&-e^{\phi/2}\sum_{i,p,q}\frac{\la_q^2(u^{q}_{pi})^2}{(1+\lambda^2_i)(1+\lambda^2_p)(1+\lambda^2_q)}
+e^{\phi/2}\sum_{\alpha,i,p}\frac{(u^{\alpha}_{pi})^2}{(1+\lambda^2_i)(1+\lambda^2_p)}\nonumber\\
&&+e^{\phi/2}\sum_{i,p,q}\frac{\la_p\la_qu^p_{qi}u^q_{pi}}{(1+\lambda^2_i)(1+\lambda^2_p)(1+\lambda^2_q)}+
e^{\phi/2}\sum_{i,p,q}\frac{\la_p\la_qu^p_{pi}u^q_{qi}}{(1+\lambda^2_i)(1+\lambda^2_p)(1+\lambda^2_q)}\nonumber\\
&=&e^{\phi/2}\sum_{i,p,q}\frac{(u^{q}_{pi})^2}{(1+\lambda^2_i)(1+\lambda^2_p)(1+\lambda^2_q)}
+e^{\phi/2}\sum_{i,p}\frac{1}{(1+\lambda^2_i)(1+\lambda^2_p)}\sum_{\alpha}(u^{\alpha}_{pi})^2\nonumber\\
&&+2e^{\phi/2}\sum_{i,p}\frac{\la_p^2(u^p_{pi})^2}{(1+\lambda^2_i)(1+\lambda^2_p)^2}+
e^{\phi/2}\sum_{i,p\neq q}\frac{\la_p\la_q(u^p_{pi}u^q_{qi}+u^p_{qi}u^q_{pi})}{(1+\lambda^2_i)(1+\lambda^2_p)(1+\lambda^2_q)}.\nonumber
\end{eqnarray}
Then the right hand side of the above equation is same to the right hand side of
the formula (3.7) in the proposition 3.1 of \cite{JXY}. Hence, we
can use their Theorem 3.1. It says
$$L_g(e^{\phi/2})\ \ \geqq \ \ K_0|B|^2\ \ \geqq\ \ \varepsilon g^{ij}(e^{\phi/2})_i(e^{\phi/2})_j.$$
Here, $B$ is the second fundamental form, and we used
$\det(g_{ij})<9$ in the above inequality. Now by Lemma \ref{sta}, we
obtain that $\phi$ is a constant. And the above inequality tells us
$B=0$, which implies $M$ is a linear subspace.

\noindent (iii) By the condition and \eqref{3.8}, we
have
\begin{eqnarray}
L_g(\phi)&=&-2\sum_{i,p,q}\frac{\la_q^2(u^{q}_{pi})^2}{(1+\lambda^2_i)(1+\lambda^2_p)(1+\lambda^2_q)}
+2\sum_{\alpha,i,p}\frac{(u^{\alpha}_{pi})^2}{(1+\lambda^2_i)(1+\lambda^2_p)}\nonumber\\
&&+2\sum_{i,p,q}\frac{\la_p\la_qu^p_{pi}u^q_{qi}}{(1+\lambda^2_i)(1+\lambda^2_p)(1+\lambda^2_q)}\nonumber\\
&\geqq&\sum_i\frac{2}{(1+\lambda^2_i)}\left(\sum_{p}\frac{\la_pu^p_{pi}}{1+\lambda^2_p}\right)^2
\ \ =\ \ \frac{1}{2}g^{ij}\phi_i\phi_j.\nonumber
\end{eqnarray}
By Lemma \ref{sta}, $\phi$ is a constant. And the above inequality
tells us that $u^{\alpha}$ are  linear functions.
\end{proof}
\begin{cor} If the normal bundle of the self-shrinking system \eqref{meq} is flat or the
codimension of the self-shrinking system \eqref{meq} is 1, then the
graph $M$ is a linear subspace.
\end{cor}
\begin{proof}
We diagonal $u^\alpha_i$ as (\ref{diag}), then $\{e_i\}$ and
$\{n_\alpha\}$ are orthonormal. Let
$h_{ij}^\alpha=\langle\overline{\nabla}_{e_i}e_j,n_\alpha\rangle$.
Then  normal bundle flat means
$$h^{\alpha}_{qi}g^{ij}h^{\beta}_{jp}=h^{\beta}_{qi}g^{ij}h^{\alpha}_{jp},\ \ \ \mathrm{for}\ \mathrm{any}\ \alpha,\beta\in\{1,\cdots,m\}.$$
Since
\begin{eqnarray}
h_{ij}^\alpha=\langle\overline{\nabla}_{e_i}e_j, n_{\alpha}\rangle
=\langle u^\beta_{ij}E_{n+\beta}, n_{\alpha}\rangle
=\dfrac{u^{\alpha}_{ij}}{(1+|D u^{\alpha}|^2)^{1/2}},\nonumber
\end{eqnarray}
the condition (iii) in Theorem \ref{10} holds, so the graph is a
linear subspace. If the codimension is 1, then the condition (iii)
in Theorem \ref{10} obviously holds. Therefore we obtain the result.
\end{proof}
\begin{remark}
 The second result of above Corollary is firstly proved
by L. Wang in \cite{EH}.
\end{remark}
Now we study the rotational symmetric manifolds corresponding to the
origin. This means
$$u^{\alpha}(x_1,\cdots, x_n)= u^{\alpha}(r),$$
where $r=\sqrt{x_1^2+\cdots+x_n^2}$. Directly calculation shows
$$u^{\alpha}_i=u^{\alpha}_r\frac{x_i}{r},\ \ \text{and}\ \
u^{\alpha}_{ij}=u^{\alpha}_{rr}\frac{x_ix_j}{r^2}+u^{\alpha}_r(\frac{\delta_{ij}}{r}-\frac{x_ix_j}{r^3}).$$
So the metric and the inverse metric become
$$g_{ij}=\delta_{ij}+|u_r|^2\frac{x_ix_j}{r^2},\
\ \text{and} \ \
g^{ij}=\delta_{ij}-\frac{|u_r|^2}{1+|u_r|^2}\frac{x_ix_j}{r^2},$$
where $|u_r|^2=\sum_{\alpha}(u^{\alpha}_r)^2$. Then, we have
\begin{eqnarray}
g^{ij}u^{\alpha}_{ij}
&=&\frac{u^{\alpha}_{rr}}{1+|u_r|^2}u^{\alpha}_{rr}+u^{\alpha}_r\frac{n-1}{r}\nonumber.
\end{eqnarray}
So the system \eqref{meq} becomes
\begin{eqnarray}
\frac{u^{\alpha}_{rr}}{1+|u_r|^2}+u^{\alpha}_r\frac{n-1}{r}&=&-u^{\alpha}+ru^{\alpha}_r.
\end{eqnarray}
\begin{prop}
Suppose a smooth solution of the self-shrinking system \eqref{meq}
is a rotation symmetry manifold. Then, it should be $\mathbb{R}^n$
plane.
\end{prop}
\begin{proof}
We let
$$\phi=\ln (1+|u_r|^2). $$ For $r\in
(0,+\infty)$, using a similar computation of
\eqref{3.6}-\eqref{3.8}, we have
\begin{eqnarray}
\frac{\phi_{rr}}{1+|u_r|^2}&=&(r-\frac{n-1}{r})\phi_r
+2\frac{n-1}{r^2}\frac{|u_r|^2}{1+|u_r|^2}+2\frac{\sum_{\alpha}(u^{\alpha}_{rr})^2}{1+|u_r|^2}\nonumber\\
&\geqq&(r-\frac{n-1}{r})\phi_r+\frac{1}{2}\frac{\phi_r^2}{1+|u_r|^2}.\nonumber
\end{eqnarray}
Let us assume that in $[0,+\infty)$, $\phi$ is not a constant. It
implies that there is some $R_0>\sqrt{2n}$, such that in $[0, R_0]$,
$\phi$ is not a constant. By the strong maximum principle, it means
that the maximum point of $\phi$ only achieves at $r=0$ or $r=R_0$.
For the rotation symmetry manifolds, $|D u| (0)=0$. In $[0,R_0]$,
suppose $\phi$ achieves its maximum value at $r=0$. Since $\phi\geqq
0$, then $\phi =0$ in $[0,R_0]$, which is a contradiction. Hence,
$\phi$ achieves its maximum point at $r=R_0$. For $0<k<1$, let
\begin{eqnarray*}\label{e4.1}
\eta(r)=\left\{\begin{matrix}1& |r|\leqq R_0\\
-k(r^2-2(n-1)\ln r -R_0^2+(n-1)\ln R_0^2)+1 & |r|\geqq
R_0\end{matrix}\right..
\end{eqnarray*}
By the same argument using in Lemma \ref{sta}, for sufficiently
small $k$ and sufficiently large $R_0$, $\eta e^{C\phi}$ only
achieves its maximum value in $r>R_0$ and $\eta>0$. Now at the
maximum point,
\begin{eqnarray}
e^{-C\phi}\frac{(\eta
e^{C\phi})_{rr}}{1+|u_r|^2}&=&\frac{\eta_{rr}}{1+|u_r|^2}+\frac{2C\eta_r
\phi_{r}}{1+|u_r|^2}+\frac{C\eta \phi_{rr}}{1+|u_r|^2}+\frac{C^2\eta
\phi^2_{r}}{1+|u_r|^2}\nonumber\\
&\geqq&\frac{-2k(1+\dfrac{n-1}{r^2})}{1+|u_r|^2}+C\eta(r-\frac{n-1}{r})\phi_r+C\eta(\frac{1}{2}-C)\frac{
\phi^2_{r}}{1+|u_r|^2}.\nonumber
\end{eqnarray}
Here, we used $\eta_r+C\eta\phi_r=0.$ Now we choose $C<1/2$, then
$$e^{-C\phi}\frac{(\eta
e^{C\phi})_{rr}}{1+|u_r|^2}\geqq
\frac{-2k(1+\dfrac{n-1}{r^2})}{1+|u_r|^2}+2k(r-\frac{n-1}{r})^2>0,$$
which is a contradiction. Then $\phi$ is a constant which implies
$u^{\alpha}=0$.
\end{proof}

\section{A sharp growth estimate in Euclidean space}
This section is composed by two parts. In the first part, we give a
sharp growth estimate for the system \eqref{meq} in Euclidean space.
In the second part, we generalize the previous result. In fact, we
will prove a slightly weak result for a class of linear elliptic
systems.

Let $T$ be an arbitrary fixed positive constant. For $(x,t)\in
\R^n\times[0,T)$, we let (see \cite{Lwang} for codimension 1 case)
\begin{eqnarray}\label{4.16}
w^\alpha(x,t)=\sqrt{T-t}\ u^\alpha(\frac{x}{\sqrt{T-t}})\quad
\text{for}\ \alpha\in\{1,\cdots,m\}.
\end{eqnarray} In $\R^n\times [0,T)$,
let
\begin{eqnarray}
\bar{g}_{ij}(x,t)&=&\delta_{ij}+\sum_{\alpha}\frac{\p w^{\alpha}}{\p x_i}\frac{\p w^{\alpha}}{\p x_j}\\
&=& \delta_{ij}+\sum_{\alpha}\frac{\p u^{\alpha}}{\p
x_i}(\frac{x}{\sqrt{T-t}})\frac{\p u^{\alpha}}{\p
x_j}(\frac{x}{\sqrt{T-t}})\ \ =\ \ g_{ij}(\f{x}{\sqrt{T-t}})\nonumber.
\end{eqnarray}
A direct calculation shows,
\begin{eqnarray}
\frac{\partial w^\alpha}{\partial
t}&=&-\frac{1}{2\sqrt{T-t}}u^{\alpha}(\frac{x}{\sqrt{T-t}})+x_i\frac{\p
u^\alpha}{\p x_i}(\frac{x}{\sqrt{T-t}})
\frac{1}{2(T-t)}\nonumber\\
&=&\frac{1}{2\sqrt{T-t}}\left(-u^{\alpha}(\frac{x}{\sqrt{T-t}})+\frac{x}{\sqrt{T-t}}\cdot Du^{\alpha}(\frac{x}{\sqrt{T-t}})\right)\nonumber\\
&=&\frac{1}{2\sqrt{T-t}}g^{ij}(\frac{x}{\sqrt{T-t}})u^\alpha_{ij}(\frac{x}{\sqrt{T-t}}),\nonumber
\end{eqnarray}
where we used the system \eqref{meq} in the last equality. Define a
heat operator,
$$\mathcal{L}_g=\frac{\partial }{\partial t}-\frac{1}{2}\sum_{i,j}\bar{g}^{ij}\frac{\partial^2}{\partial x_i\partial x_j},$$
where $(\bar{g}^{ij})$ is the inverse matrix of $(\bar{g}_{ij})$.
Note that
$$w^{\alpha}_{ij}\ \ =\ \ \frac{1}{\sqrt{T-t}}u^\alpha_{ij}(\frac{x}{\sqrt{T-t}}).$$
Then we obtain
\begin{eqnarray}\label{4.17}
 \mathcal{L}_g(w^{\alpha})=0.
\end{eqnarray}
In $\R^n\times [0,T)$, we let
$$\eta(x,t)=1-|x|^2-3nt,\ \ \text{ and } \ \ |w|^2=\sum_{\alpha}(w^\alpha)^2.$$
Using \eqref{4.17}, we have
\begin{equation}\label{5.1}
\mathcal{L}_g|w|^2\ \ =\ \ 2\sum_\alpha
w^\alpha\mathcal{L}_gw^\alpha-\sum_{\alpha,i,j}\bar{g}^{ij}w_i^\alpha
w_j^\alpha\ \ =\ \ -\sum_{\alpha,i,j} \bar{g}^{ij}w_i^\alpha w_j^\alpha,
\end{equation}
and
\begin{eqnarray}\label{5.2}
\mathcal{L}_g\eta &=&-3n+\sum_i \bar{g}^{ii}.
\end{eqnarray}

Now, we are in the position to give the proof of Theorem \ref{thm1}.
\vspace{1mm}

\noindent {\it Proof of Theorem \ref{thm1}.} For any fixed $\rho
>0$, in \eqref{4.16}, we take
$$T=\frac{1}{12n}+\frac{1}{\rho^2}.$$ In $\R^n\times [0,T)$, we define
a function $$\phi = \eta |w|^2. $$ Assume that  $\phi$ achieves its
maximum value at some point $p=(x_0,t_0)$ in the set $\R^n\times [0,
\dfrac{1}{12n}]$, where $\eta(x_0,t_0)>0$. By the definition of
$\eta(x,t)$, we have
\begin{eqnarray}\label{5.7}
\frac{1}{2\rho^2}\sup_{|x|\leqq\rho/2}|u(x)|^2&=&\frac{1}{2}\sup_{|x|\leqq 1/2}\frac{\sum_\alpha|u^\alpha(\rho x)|^2}{\rho^2}\\
&\leqq& \sup_{|x|\leqq 1/2}\left(\sum_\alpha|\frac{u^\alpha(\rho
x)}{\rho}|^2\eta(x,\frac{1}{12n})\right)\nonumber\\
&\leqq&\sup_{x\in \mathbb{R}^n}\phi|_{t=\frac{1}{12n}}.\nonumber
\end{eqnarray}
If $\phi(p)\leqq 1$, then
\begin{equation}\aligned\label{5.6}
\sup_x\phi|_{t=\frac{1}{12n}}\ \ \leqq\ \ \phi(p)\ \ \leqq \ \ 1.
\endaligned
\end{equation}
Combining (\ref{5.7}), (\ref{5.6}) and using the arbitrary choice of
$\rho$, we get the theorem.

Since $\phi(p)\geqq 1$ and $\eta\leqq 1$, we get $|w(p)|\geqq1$. If
$t_0>0$, then we have
\begin{equation}\aligned\label{5.3}
2\sum_\alpha w^\alpha\eta D w^\alpha+|w|^2 D \eta=0
\endaligned
\end{equation}
at the point $p$, where $D$ only takes derivatives to space
directions as before. Thus we have
\begin{equation}\aligned\label{5.4}
0\leqq\ &\mathcal{L}_g(\phi)\\
=&\ \eta\mathcal{L}_g|w|^2+|w|^2\mathcal{L}_g\eta-2\sum_{\alpha,i,j}\bar{g}^{ij}w^\alpha w^\alpha_i\eta_j\\
=&\ -\eta \sum_{\alpha,i,j}\bar{g}^{ij}w_i^\alpha w_j^\alpha+|w|^2(-3n+\sum_{i}\bar{g}^{ii})+2\sum_{\alpha,i,j}\bar{g}^{ij}w^\alpha w^\alpha_i(\sum_{\beta}\frac{1}{|w|^2}2w^\beta\eta w^\beta_j)\\
\leqq&\ -\eta \sum_{\alpha,i,j}\bar{g}^{ij}w_i^\alpha w_j^\alpha+|w|^2(-3n+\sum_{i}\bar{g}^{ii})+4\eta \sum_{\alpha,i,j}\bar{g}^{ij}w_i^\alpha w_j^\alpha\\
\leqq&\ 3\eta(n-\sum_{i}\bar{g}^{ii})+(-3n+\sum_i\bar{g}^{ii})\\
=&\ -2\sum_i\bar{g}^{ii}+3(\eta-1)(n-\sum_i\bar{g}^{ii}).
\endaligned
\end{equation}
Here we have used \eqref{5.1}, \eqref{5.2} and \eqref{5.3} in the
third step and the Cauchy inequality in the forth step. In fact,
there exists an orthonormal matrix $(P_{ij})_{n\times n}$, such that
$\bar{g}_{ij}=P_{ik}\theta_k P_{jk}$, where $\theta_k>0$ at the
point $p$. Then
\begin{eqnarray}
\sum_{\alpha,\beta,i,j}\bar{g}^{ij}w^{\alpha}w^{\alpha}_iw^{\beta}w^{\beta}_j
&=&
\sum_{\alpha,\beta,i,j,k}\theta_k w^{\alpha}w^{\alpha}_iP_{ik}w^{\beta}w^{\beta}_jP_{jk}
\ \ = \ \
\sum_{k}\theta_k (\sum_{\alpha,i}w^{\alpha}w^{\alpha}_iP_{ik})^2\nonumber\\
&\leqq&
\sum_{k}\theta_k\left(\sum_\alpha(w^\alpha)^2\sum_\alpha(\sum_{i}w^{\alpha}_iP_{ik})^2\right)
\ \ = \ \
|w|^2\sum_{\alpha,i,j}\bar{g}^{ij}w^{\alpha}_iw^{\alpha}_j.\nonumber
\end{eqnarray}
Since $0<\bar{g}^{ii}\le1$ and $\eta\le1$ by the
definition of $\eta$, from (\ref{5.4}), we obtain $\eta(p)\geq1$
which implies $t_0=0$. We have a contradiction. Anyway, $\phi$
achieves its maximum value at $t=0$, namely,
\begin{equation}\aligned\label{5.8}
\sup_{x\in \R^n}\phi|_{t=\frac{1}{12n}}\ \ \leqq\ \ \sup_{x\in
\R^n}\phi|_{t=0}.
\endaligned
\end{equation}
Combining (\ref{5.7}) and (\ref{5.8}), we get
\begin{eqnarray}\label{5.5}
\frac{1}{2\rho^2}\sup_{|x|\leqq\rho/2}|u(x)|^2 & \leqq  &
\sup_{x\in \mathbb{\R}^n}\phi|_{t=\frac{1}{12n}}\nonumber\\
&\leqq&\sup_{x\in \mathbb{R}^n}\left((1-|x|^2)
T|u(\frac{x}{\sqrt{T}})|^2\right)\nonumber.
\end{eqnarray}
Then we obtain
\begin{eqnarray}
\sup_{|x|\leq\rho/2}|u(x)|^2&\leqq& 2\rho^2\sup_{|x|\leqq 1}
T|u(\frac{x}{\sqrt{T}})|^2\nonumber\\
&\leqq&(\frac{\rho^2}{6n}+2)
\sup_{|x|\leq2\sqrt{3n}}|u(x)|^2.\nonumber
\end{eqnarray}
Using the arbitrary choice of $\rho$, we obtain the theorem. \qed
\begin{remark}
Assume that $u^{\alpha}$ is a solution of the system \eqref{gel}
satisfying  \eqref{gelcond} and \eqref{gelcond2} for $t=1$. Let
$\bar{a}^{ij}(x,t)=a^{ij}(\dfrac{x}{\sqrt{T-t}})$,
$w^\alpha(x,t)=\sqrt{T-t}\ u^\alpha(\dfrac{x}{\sqrt{T-t}})$ and
operator $\mathcal{L}_a=\dfrac{\partial }{\partial
t}-\dfrac{1}{2}\sum_{i,j}\bar{a}^{ij}\dfrac{\partial^2}{\partial
x_i\partial x_j}$ for any fixed $T>0$. Then
$\mathcal{L}_aw^\alpha=0$. Let $\eta=1-|x|^2-3\sigma t$. Using
almost the same proof of the above theorem, we know that $|u|$ is
also linear growth.
\end{remark}

\begin{cor}
The mean curvature vector is also linear growth. Namely, there
exists a positive constant $C$ depending only on $n$ and
$\sup_{|x|\leq2\sqrt{3n}}|u(x)|$, such that
\begin{eqnarray}
|H(x)|&\leqq&C(1+|x|).
\end{eqnarray}
\end{cor}
\begin{proof}
By formula \eqref{2.2}, we have
$$H^{\alpha}=\frac{-u^{\alpha}+x\cdot D u^{\alpha}}{\sqrt{1+|D
u^{\alpha}|^2}}.$$ Then using Theorem \ref{thm1} and the Cauchy
inequality, we obtain
\begin{eqnarray}
|H^{\alpha}(x)|&\leqq &|u^{\alpha}(x)|+|x|\ \ \leqq \ \
C(1+|x|),\nonumber
\end{eqnarray}
where $C$ is a positive constant depending only on $n$ and
$\sup_{|x|\leq2\sqrt{3n}}|u(x)|$. It implies the estimate.
\end{proof}

Now let's generalize Theorem \ref{thm1} to the system \eqref{gel}
satisfying \eqref{gelcond} and \eqref{gelcond2}. Firstly, we give
some preliminary results. For $s> 0$, we denote
\begin{eqnarray}
g(s)=\frac{1}{s+1}(\frac{2s}{s+1})^s.
\end{eqnarray}
Calculating its derivative, we have
$$\frac{g'(s)}{g(s)}=\ln \frac{2s}{s+1}>0,\ \  \text{for} \ \ s>1.$$ And $g(1)=1/2$,
$g(+\infty)=+\infty$. So, there is only one $s_0$ satisfying
$$g(s_0)=1.$$ In fact, explicit calculation shows that
$3.4<s_0<3.5$.

\begin{lemma}\label{8}
For any $s\geqq s_0$, there is some $1<\zeta<2$ satisfying
\begin{eqnarray}\label{4.02}
\frac{2}{2-\zeta}&\leqq &\zeta^s.
\end{eqnarray}
\end{lemma}
\begin{proof}
Let
$$f(\zeta)=\zeta^{s+1}-2\zeta^s+2.$$ Then the derivative of $f$ is
$$f'(\zeta)=(s+1)\zeta^{s-1}(\zeta-\frac{2s}{s+1}).$$  We see that
$$\zeta=\frac{2s}{s+1}$$ is a local minimum value of $f$. Note that
\begin{eqnarray}
f(\frac{2s}{s+1})=2[1-\frac{1}{s+1}(\frac{2s}{s+1})^s]=2(1-g(s))\nonumber.
\end{eqnarray}
Since $s\geqq s_0$, then $g(s)\geqq 1$. We have $f(\zeta)\leqq 0$,
which implies \eqref{4.02}.
\end{proof}
For $s\geqq s_0$, denote
\begin{eqnarray}\label{4.03}
\theta=\sqrt{\frac{s}{s+1}},\ \ k=\sqrt{2}\theta,\ \
R_0^2=\max\{r_0,\frac{n\sigma+1}{2}\frac{k^2}{k^2-1}\}.
\end{eqnarray}
Here, $r_0$ is a radius taking in \eqref{gelcond2}.  For any fixed
$R>R_0$, let
\begin{eqnarray}\label{4.04}
\bar{R}=R/\theta>R,\ \ |u|=\sum_{\alpha}(u^{\alpha})^2.
\end{eqnarray}
Define two functions
\begin{eqnarray}\label{4.05}
\eta=\bar{R}^2-|x|^2, \ \text{and} \ \ \phi=\eta|u|^2.
\end{eqnarray}
Then $\phi$ achieves its maximum value in the set
$\{x\in\mathbb{R}^n;\eta>0\}$. We assume the maximum point to be
$p$. At $p$,
\begin{eqnarray}\label{4.06}
\phi_i&=&\eta_i|u|^2+2\eta\sum_{\alpha}u^{\alpha}u^{\alpha}_i\ \ =\
\ 0.
\end{eqnarray}
Then
\begin{eqnarray}
a^{ij}\phi_{ij}
&=&a^{ij}\eta_{ij}|u|^2+4\sum_{\alpha}u^{\alpha}a^{ij}\eta_i
u^{\alpha}_j+2\eta
\sum_{\alpha}a^{ij}u^{\alpha}_iu^{\alpha}_j+2\eta\sum_{\alpha}u^{\alpha}a^{ij}u^{\alpha}_{ij}.\nonumber
\end{eqnarray}
Using \eqref{gel}, \eqref{4.05} and \eqref{4.06}, we get
\begin{eqnarray}
a^{ij}\phi_{ij}|u|^2
&=&-2a^{ij}\delta_{ij}|u|^4-8\sum_{\alpha,\beta}\eta
u^{\alpha}u^{\beta}a^{ij}u^{\beta}_iu^{\alpha}_j
+2\eta|u|^2\sum_{\alpha}a^{ij}u^{\alpha}_iu^{\alpha}_j\nonumber\\&&+2\eta|u|^2\sum_{\alpha}
u^{\alpha}(-u^{\alpha}+x\cdot D u^{\alpha})\nonumber.
\end{eqnarray}
Using the Cauchy inequality,
\begin{eqnarray}\label{cie}
2a^{ij}u^{\alpha}u^{\beta}u^{\alpha}_iu^{\beta}_j&\leqq&
a^{ij}u^{\alpha}u^{\alpha}u^{\beta}_iu^{\beta}_j+a^{ij}u^{\alpha}_iu^{\alpha}_ju^{\beta}u^{\beta},\nonumber
\end{eqnarray}
then at $p$, by \eqref{4.06}, we get
\begin{eqnarray}
0&\geqq
&-2\sum_{i}a^{ii}|u|^4-8\eta|u|^2\sum_{\alpha}a^{ij}u^{\alpha}_iu^{\alpha}_j
+2\eta|u|^2\sum_{\alpha}a^{ij}u^{\alpha}_iu^{\alpha}_j\nonumber\\&&+2\eta|u|^2
[-|u|^2+\frac{1}{2}x\cdot D|u|^2]\nonumber\\
&=&|u|^2[-2\sum_{i}a^{ii}|u|^2
-6\eta\sum_{\alpha}a^{ij}u^{\alpha}_iu^{\alpha}_j -2\eta|u|^2-x\cdot
D \eta |u|^2]\nonumber.
\end{eqnarray}
If at $p$, $|u|=0$, then in $B_R$, $u^{\alpha}=0$. Hence, in $B_R$,
it is obviously polynomial growth. So we can assume at $p$, $|u|\neq 0$.
Then
\begin{eqnarray}
6\eta\sum_{\alpha}a^{ij}u^{\alpha}_iu^{\alpha}_j&\geqq
&-2\sum_{i}a^{ii}|u|^2 -2\eta|u|^2-x\cdot D \eta
|u|^2\nonumber\\
&\geqq&(4|x|^2-2n\sigma-2\bar{R}^2)|u|^2\nonumber.
\end{eqnarray}
So, we obtain
\begin{eqnarray}
(4|x|^2-2n\sigma-2\bar{R}^2)\phi&\leqq&
6\eta^2\sum_{\alpha}a^{ij}u^{\alpha}_iu^{\alpha}_j\ \ \leqq \ \
6c\bar{R}^{2\tau+2}\nonumber.
\end{eqnarray}
Hence, we have two cases. The first case is
$$p\in \{x\in \mathbb{R}^n;2|x|^2-n\sigma-\bar{R}^{2}\geqq 1 \}.$$ Then
$$\phi(p)\ \ \leqq \ \ 3c\bar{R}^{2\tau+2}.$$ Since $\phi$ achieves  its maximum
value at $p$,  we have, in $B_R$,
\begin{eqnarray}
3c \bar{R}^{2\tau+2}&\geqq&(\bar{R}^2-|x|^2)|u|^2\ \ \geqq\ \
(\bar{R}^2-R^2)|u|^2\nonumber.
\end{eqnarray}
Using \eqref{4.04},  we obtain
\begin{eqnarray}\label{4.010}
\sup_{B_R} |u|^2&\leqq &\frac{3c
R^{2\tau}}{\theta^{2\tau}(1-\theta^2)}.
\end{eqnarray}
The other case is $$p\in
\{x\in\mathbb{R}^n;2|x|^2-n\sigma-\bar{R}^2\leqq 1 \}.$$ Then, we
have
\begin{eqnarray}\label{4.011}
\sup_{B_R}\phi&\leqq &\sup_{B_{r}}\phi.
\end{eqnarray}
Here
\begin{eqnarray}
r&=&\sqrt{(\bar{R}^2+n\sigma+1)/2}\ \ =\ \
\sqrt{\frac{R^2}{k^2}+\frac{n\sigma+1}{2}}.\nonumber
\end{eqnarray}
By \eqref{4.011}, we have
\begin{eqnarray}
(\frac{R^2}{\theta^2}-R^2)\sup_{B_R} |u|^2&\leqq&
\frac{R^2}{\theta^2}\sup_{B_r} |u|^2\nonumber.
\end{eqnarray}
Using Lemma \ref{8}, we have
\begin{eqnarray}\label{4.013}
\sup_{B_R} |u|^2&\leqq& \frac{2}{2-k^2}\sup_{B_r} |u|^2\ \ \leqq \ \
k^{2s}\sup_{B_r} |u|^2.
\end{eqnarray}
Combining \eqref{4.010} and \eqref{4.013}, we obtain the following
Lemma.
\begin{lemma}\label{9} Let $s\geqq s_0$. For every $R>R_0$, any solution of system \eqref{gel} satisfying \eqref{gelcond} and \eqref{gelcond2} should
satisfy one of the  two inequalities: \eqref{4.010} or
\eqref{4.013}.
\end{lemma}
So, we have the following growth estimate.
\begin{thm}\label{thm15}
Assume that $u^{\alpha}$ is a solution of the system \eqref{gel}
satisfying  \eqref{gelcond} and \eqref{gelcond2}.  Then $|u|$ is
polynomial growth. Namely, for $s\geqq s_0$, the solution
$u^{\alpha}$ have the estimate,
\begin{eqnarray}\label{4.014}
u^{\alpha}(x)\leqq C
(1+\sup_{B_{\sqrt{k^2+1}R_0}}|u|)(1+|x|^{\max\{s,\tau\}}),
\end{eqnarray}
where $C$ depending on $\sigma,c,n,s,\tau$ and $r_0$.
\end{thm}
\begin{proof}
Let $$R=|x|>R_0.$$ There is a nonnegative integer $m_0$, such that
\begin{eqnarray}\label{4.015}
R_0^2\ \ \leqq\ \ \frac{R^2}{k^{2m_0}}\ \ \leqq \ \ k^2R_0^2.
\end{eqnarray}
For $m\geqq 2$, denote
\begin{eqnarray}\label{4.016}
R_m^2&=&
\frac{R^2}{k^{2(m-1)}}+\frac{n\sigma+1}{2}[1+\frac{1}{k^2}+\cdots+\frac{1}{k^{2(m-2)}}].
\end{eqnarray}
Let $R_1=R$. Obviously, for $1\leqq m\leqq m_0+1$,
$$R_{m}>R_0.$$ It implies  Lemma \ref{9} is applicable for $B_{R_m}$. Hence, it will appear two cases. The first is that
\eqref{4.013} holds for every $1\leqq m\leqq m_0$. The second is
that there is an integer $1\leqq m_1\leqq m_0$, such that, for any
integer $1\leqq m\leqq m_1-1$, \eqref{4.013} holds and \eqref{4.010}
holds in $B_{R_{m_1}}$. In the first case, for $1\leqq m\leqq m_0$,
we have
\begin{eqnarray}\label{4.017}
\sup_{B_{R_m}}|u|^2 &\leqq& k^{2s} \sup_{B_{R_{m+1}}}|u|^2.
\end{eqnarray}
Iterating \eqref{4.017}, we get
\begin{eqnarray}
\sup_{B_{R}}|u|^2&=&\sup_{B_{R_1}}|u|^2 \ \ \leqq \ \
k^{2sm_0}\sup_{B_{R_{m_0+1}}}|u|^2\nonumber.
\end{eqnarray}
By \eqref{4.04}, \eqref{4.015} and \eqref{4.016}, we have
\begin{eqnarray}
R_{m_0+1}^2&=&\frac{R^2}{k^{2m_0}}+\frac{n\sigma+1}{2}[1+\frac{1}{k^2}+\cdots+\frac{1}{k^{2(m_0-1)}}]\nonumber\\
&=&\frac{R^2}{k^{2m_0}}+\frac{n\sigma+1}{2}\frac{1-1/k^{2m_0}}{1-1/k^2}\nonumber\\
&\leqq& k^2 R^2_0+\frac{n\sigma+1}{2}\frac{1}{1-1/k^2}\nonumber\\
&\leqq&(k^2+1)R_0^2\nonumber.
\end{eqnarray}
Combining the above two inequalities and \eqref{4.015},  we obtain
\begin{eqnarray}
\sup_{B_{R}}|u|^2&\leqq&
\frac{R^{2s}}{R_0^{2s}}\sup_{B_{\sqrt{k^2+1}R_0}}|u|^2.
\end{eqnarray}
In the second case, \eqref{4.017} holds for $1\leqq m \leqq m_1-1$,
then, similar, we have
\begin{eqnarray}
\sup_{B_{R}}|u|^2&\leqq& k^{2s(m_1-1)}\sup_{B_{R_{m_1}}}|u|^2\ \
\leqq\ \  Ck^{2s(m_1-1)}R_{m_1}^{2\tau}.\nonumber
\end{eqnarray}
Here, $C$ is a constant depending on $c$ and $s$. By \eqref{4.015},
we have
$$k^{2(m_1-1)}\ \ \leqq \ \ \frac{R^2}{R_0^2k^{2(m_0+1-m_1)}}\ \ \leqq \ \ \frac{R^2}{R_0^2}.$$
Combining the above two inequalities, for $s>\tau$, we get
\begin{eqnarray}
\sup_{B_{R}}|u|^2
&\leqq&Ck^{2(s-\tau)(m_1-1)}\{k^{2(m_1-1))}[\frac{R^2}{k^{2(m_1-1)}}
+\frac{n\sigma+1}{2}(1+\cdots+\frac{1}{k^{2(m_1-2)}})]\}^{\tau}\nonumber\\
&\leqq& C \frac{R^{2(s-\tau)}}{R_0^{2(s-\tau)}}[R^2+k^{2(m_1-1)}\frac{n\sigma+1}{2}\frac{1}{1-1/k^2}]^{\tau}\nonumber\\
&\leqq&C R^{2(s-\tau)}[R^2+\frac{R^2}{R_0^2}R_0^2]^{\tau}\nonumber\\
&\leqq& C R^{2s}\nonumber.
\end{eqnarray}
For $s\leqq \tau$, the above second inequality  becomes
\begin{eqnarray}
\sup_{B_{R}}|u|^2
&\leqq&C[R^2+k^{2(m_1-1)}\frac{n\sigma+1}{2}\frac{1}{1-1/k^2}]^{\tau}\
\ \leqq \ \ CR^{2\tau}\nonumber.
\end{eqnarray}
 Combining the above two inequalities, for $|x|>R_0$, we obtain
\begin{eqnarray}
|u|^2(x)&\leqq&C(1+\sup_{B_{\sqrt{k^2+1}R_0}}|u|^2)|x|^{2\max\{s,\tau\}}.\nonumber
\end{eqnarray}
It implies \eqref{4.014}.
\end{proof}
\begin{remark}
The method using in the proof of Theorem \ref{thm15} also can be
used to obtain a growth estimate of $|u|$ for the system \eqref{meq}
in Euclidean space. But it is not sharp.
\end{remark}

\section{A Bernstein type result of space-like  self-shrinking graph in pseudo-Euclidean space}

\begin{prop}
In pseudo-Euclidean space with index $m$, if the eigenvalues of the
metric matrix $(g_{ij})$ have a positive low bound, then $M$ is a
linear subspace.
\end{prop}
\begin{proof}
In pseudo-Euclidean space, using a similar argument for the volume
element function, we obtain
\begin{eqnarray}\label{5.11111}
&&L_g(-\phi)\\
&=&2\sum_{\alpha,\beta}g^{ij}g^{mk}g^{nl}u^{\alpha}_{mi}u^{\alpha}_nu^{\beta}_{kj}u^{\beta}_l
-2\sum_{\alpha,\beta}g^{ij}g^{mk}g^{nl}u^{\alpha}_{ni}u^{\alpha}_mu^{\beta}_{kj}u^{\beta}_l\nonumber\\
&&+2\sum_{\alpha}g^{ij}g^{mn}u^{\alpha}_{mi}u^{\alpha}_{nj}\nonumber\\
&=&\sum_{\alpha,\beta}g^{ij}g^{mk}g^{nl}(u^{\alpha}_{mi}u^{\alpha}_n-u^{\alpha}_{ni}u^{\alpha}_m)
(u^{\beta}_{kj}u^{\beta}_l-u^{\beta}_{lj}u^{\beta}_k)+2\sum_{\alpha}g^{ij}g^{mn}u^{\alpha}_{mi}u^{\alpha}_{nj}\nonumber\\
&\geqq&2\sum_{\alpha}g^{ij}g^{mn}u^{\alpha}_{mi}u^{\alpha}_{nj}\nonumber\\
&\geqq& \varepsilon g^{ij}(-\phi)_i(-\phi)_j\nonumber.
\end{eqnarray}
By Lemma \ref{sta}, we get that $-\phi$ is a constant. Then the
first inequality of \eqref{5.11111} implies $u^{\alpha}_{ij}=0$. This is
the result.
\end{proof}

Now we continue to consider spacelike self-shrinkering graph $M$
described by \eqref{x,u} in pseudo-Euclidean space with index $m$.
In what following, we denote the pseudo-inner product
$\langle\cdot,\cdot\rangle$ by
$$\langle X,X\rangle=\sum_ix_i^2-\sum_\alpha (y^\alpha)^2,\ \ \mathrm{for\ each}\ X=(x_1,\cdots,x_n;y^1,\cdots,y^m).$$
We write $|X|^2=\langle X,X\rangle$ for $X\in\R^{n+m}$. We assume that
the metric of $M$ is the induced metric.

For any $p\in\R^n$, we choose a coordinate system as (\ref{diag}).
Then, there are orthonormal vectors $\{e_i(p)\}_{i=1}^n\subset
T_pM$, $\{n_\alpha(p)\}_{\alpha=1}^m\subset N_pM$ defined as
\begin{equation}\aligned\label{4.1}
e_i(p)=\frac{1}{\sqrt{1-\lambda_i^2}}(E_i+\lambda_iE_{n+i}), \qquad
n_{\alpha}(p)=\frac{1}{\sqrt{1-\lambda_{\alpha}^2}}(\lambda_{\alpha}E_{\alpha}+E_{n+\alpha}),
\endaligned
\end{equation}
where spacelike implies $|\lambda_i|<1$ for every $i$. Then we can
choose a normal frame $\{e_i\}_{i=1}^n\subset \Gamma(TM)$ and an
orthonormal frame $\{n_\alpha\}_{\alpha=1}^m\subset \Gamma(NM)$
locally, such that $e_i(p)$ and $n_\alpha(p)$ is defined by
\eqref{4.1}. We will use these notations in the following, which are
different from the definitions in  section 2.

The second fundamental form and mean curvature
vector are,
$$B_{ij}=\overline{\nabla}^N_{e_i}e_j=\sum_\alpha h_{ij}^{\alpha}
n_\alpha,\ \  \text{and} \ \
H=\sum_i\overline{\nabla}^N_{e_i}e_i.$$ Denote the
standard $n$-form in $\R^n$ by $$dx=dx_1\wedge\cdots\wedge dx_n.$$
Then, we define a function,
$$*dx=dx(e_1,\cdots,e_n).$$
Then, at the point $p$, we have
\begin{equation}\aligned\label{4.2}
\langle\overline{\nabla}_{e_i}\overline{\nabla}_{e_i}e_j,e_j\rangle\
\ =\ \
\overline{\nabla}_{e_i}\langle\overline{\nabla}_{e_i}e_j,e_j\rangle
-\langle\overline{\nabla}_{e_i}e_j,\overline{\nabla}_{e_i}e_j\rangle\
\ =\ \ \sum_\alpha(h_{ij}^\alpha)^2.
\endaligned
\end{equation}
Since
$\overline{\nabla}_{e_i}e_j-\overline{\nabla}_{e_j}e_i=[e_i,e_j]\in
TM$, then we have, at $p$,
\begin{equation}\aligned\nonumber
&\langle\overline{\nabla}_{e_i}\overline{\nabla}_{e_i}e_j,n_\alpha\rangle-\langle\overline{\nabla}_{e_i}\overline{\nabla}_{e_j}e_i,n_\alpha\rangle\\
=\ \ &\langle\overline{\nabla}_{e_i}[e_i,e_j],n_\alpha\rangle\ \ =\ \ -\langle[e_i,e_j],\overline{\nabla}_{e_i}n_\alpha\rangle\ \ =\ \ 0.
\endaligned
\end{equation}
Since $\R^{n+m}$ is flat, using  \eqref{1.2} and the above equality,
we have, at $p$,
\begin{eqnarray}\label{4.3}
&&\sum_i\langle\overline{\nabla}_{e_i}\overline{\nabla}_{e_i}e_j,n_\alpha\rangle\\
&=&
\sum_i\langle\overline{\nabla}_{e_i}\overline{\nabla}_{e_j}e_i,n_\alpha\rangle\
\ \ \  =\ \
\sum_i\langle\overline{\nabla}_{e_j}\overline{\nabla}_{e_i}e_i,n_\alpha\rangle\
\ \ = \ \ \langle\overline{\nabla}_{e_j}
H,n_\alpha\rangle\nonumber\\ & = &
-\left\langle\overline{\nabla}_{e_j}(X-\sum_k\langle X,e_k\rangle
e_k),n_\alpha\right\rangle\nonumber\\
&=&\sum_k\langle X,e_k\rangle\langle\overline{\nabla}_{e_j}
e_k,n_\alpha\rangle \ \ = \ \ -\sum_k\langle X,e_k\rangle
h_{jk}^{\alpha}.\nonumber
\end{eqnarray}
By the definition of $dx$ and (\ref{4.1}), we have, at $p$,
\begin{eqnarray}\label{4.4}
&&\nabla_{e_i}*dx \ \ = \ \ \sum_jdx(e_1,\cdots,\overline{\nabla}_{e_i}e_j,\cdots,e_n)\\
&=&\sum_jh_{ij}^\alpha
dx(e_1,\cdots,\underbrace{n_\alpha}_j,\cdots,e_n) \ \ =\ \
\sum_jh^{j}_{ij}\lambda_j*dx.\nonumber
\end{eqnarray}
Let us calculate the Laplace of $*dx$, which is similar to the
minimal submanifolds case (see \cite{Wang2}). Combining
(\ref{4.1})-(\ref{4.3}), we obtain
\begin{eqnarray}\label{4.5}
&&\Delta*dx\\
&=&\sum_{i,j\neq
k}dx(e_1,\cdots,\overline{\nabla}_{e_i}e_j,\cdots,\overline{\nabla}_{e_i}e_k,\cdots,e_n)
+\sum_{i,j}dx(e_1,\cdots,\overline{\nabla}_{e_i}\overline{\nabla}_{e_i}e_j,\cdots,e_n)\nonumber\\
&=&\sum_{i,j\neq k}h_{ij}^\alpha h_{ik}^\beta
dx(e_1,\cdots,\underbrace{n_\alpha}_j,\cdots,\underbrace{n_\beta}_k,\cdots,e_n)
+\sum_{i,j}\langle\overline{\nabla}_{e_i}\overline{\nabla}_{e_i}e_j,e_j\rangle dx(e_1,\cdots,e_n)\nonumber\\
&&-\sum_{i,j}\langle\overline{\nabla}_{e_i}\overline{\nabla}_{e_i}e_j,n_\alpha\rangle dx(e_1,\cdots,\underbrace{n_\alpha}_j,\cdots,e_n)\nonumber\\
&=&\sum_{i,j\neq k}\lambda_j\lambda_kh_{ij}^jh_{ik}^k*dx-\sum_{i,j\neq k}\lambda_j\lambda_kh_{ij}^kh_{ik}^j*dx+\sum_{i,j,\alpha}(h_{ij}^{\alpha})^2*dx\nonumber\\
&&+\sum_{j,k}\langle X,e_k\rangle h_{jk}^{j}\lambda_j*dx\nonumber.
\end{eqnarray}
We define a second order differential operator(see \cite{CM} for Euclidean space),
$$P\ \ =\ \ \Delta-\langle X,\nabla\cdot\rangle.$$
Using the Cauchy inequality, since $|\lambda_i|< 1$, we have
\begin{equation}\aligned\label{4.0}
\left|\sum_{i,j\neq k}\lambda_j\lambda_kh_{ij}^kh_{ik}^j\right|\ \ \leqq\ \
\frac{1}{2}\sum_{i,j\neq k}((h_{ij}^k)^2+(h_{ik}^j)^2)\ \ =\ \
\sum_{i,j\neq k}(h_{ij}^k)^2.
\endaligned
\end{equation}
Then (\ref{4.5}) becomes
\begin{eqnarray}\label{4.6}
P(*dx)&=&\left(\sum_{i,j\neq k}\lambda_j\lambda_kh_{ij}^jh_{ik}^k-\sum_{i,j\neq k}\lambda_j\lambda_kh_{ij}^kh_{ik}^j+\sum_{i,j,\alpha}(h_{ij}^{\alpha})^2\right)*dx\\
&\geqq&\left(\sum_{i,j\neq k}\lambda_j\lambda_kh_{ij}^jh_{ik}^k+\sum_{i,k}\lambda_i^2(h_{ik}^{i})^2\right)*dx\nonumber\\
&=&\left(\sum_{i,j,k}\lambda_j\lambda_kh_{ij}^jh_{ik}^k\right)*dx \
\ =\ \ \frac{|\nabla*dx|^2}{*dx}.\nonumber
\end{eqnarray}

We are in the position to give the proof of Theorem \ref{thm2}.
\vspace{1mm}

\noindent{\it Proof of Theorem \ref{thm2}.} Denote the induced
metric in $M$ by $g=\sum_{i,j}g_{ij}dx_idx_j$, seeing \eqref{pmtc},
and $\det g=\det(g_{ij})$. Denote weighted function $\rho$ by
$$\rho=\exp(-\frac{|X|^2}{2})=\exp(-\frac{|x|^2-|u(x)|^2}{2}),$$
and the volume element of $M$ by $$d\mu=\sqrt{\det
g}dx_1\wedge\cdots\wedge dx_n=\sqrt{\det g}dx.$$ Then, for any local
orthonormal frame $\{e_i\}_{i=1}^n$ of tangent bundle $TM$, we have
$$*dx=\f1{\sqrt{\det g}},\ \ \text{ and } \ \ d\mu(e_1,\cdots,e_n)=1.$$
Since $2\langle X,\nabla\cdot\rangle=\langle\nabla
|X|^2,\nabla\cdot\rangle$, then the operator $P$ is invariant under
orthonormal transformations from $\R^n$ to $\R^n$ and from $\R^m$ to
$\R^m$. Hence, (\ref{4.6}) holds in the whole  $M$. Let
$f=\dfrac{1}{\sqrt{\det g}}$ and $\eta\in C_c^\infty(M)$. We will
determine $\eta$ later.  By (\ref{4.6}), we have
\begin{equation}\aligned\label{4.7}
\int_M\frac{|\nabla f|^2}{f}&\eta^2\rho d\mu\ \ \leqq\ \ \int_M P(f)\eta^2\rho d\mu\ \ =\ \ \int_M \mathrm{div}(\rho\nabla f)\eta^2 d\mu\\
&=\ \-2\int_M\eta(\nabla f\cdot\nabla\eta)\rho d\mu\ \
\leqq\ \ \frac{1}{2}\int_M\frac{|\nabla f|^2}{f}\eta^2\rho
d\mu+2\int_M|\nabla\eta|^2f\rho d\mu.
\endaligned
\end{equation}
Here, 'div' is the divergence of $M$. Then
\begin{equation}\aligned\label{4.8}
\int_M\frac{|\nabla f|^2}{f}\eta^2\rho d\mu \ \ \leqq \ \
4\int_{\R^n}|\nabla\eta|^2\rho dx \ \ = \ \
4\int_{\R^n}g^{ij}\eta_i\eta_j\exp(-\frac{|x|^2-|u(x)|^2}{2})dx.
\endaligned
\end{equation}
 On the other hand, there exists a constant $\kappa$ satisfying
 $0<\kappa<1$, such that $|\lambda_i|\leq\kappa$ for any $x\in
 \overline{B_1(0)}\subset\mathbb{R}^n,i\in\{1,\cdots,n\}$. If
 $|u(x)|\neq0$, using $u(0)=0$, then there exists a nonnegative
 number $a<1$, such that, $u(sx)\neq0$ for any $s\in(a,1]$ and $u(ax)=0$.
 Hence, for $p\in[ax,x]$,
 \begin{eqnarray}\label{4.10}
 \left|\sum_ix_i\dfrac{\p}{\partial x_i}(\sum_{\alpha}
 (u^\alpha)^2)\right|&=&2\left|\sum_{i,\alpha}x_iu^\alpha
 u^\alpha_i\right| \\
 &\leqq&  2\left(\sum_{\alpha}(u^{\alpha})^2\right)^{1/2}\left(\sum_{\alpha,i}x_i^2(u^{\alpha}_i)^2\right)^{1/2}\nonumber\\
 &\leqq&2\kappa|x||u|.\nonumber
 \end{eqnarray}
 Here we use the fact that every diagonal entry of a positive matrix is no large than the maximum eigenvalue of this matrix. For any $x\in \overline{B_1(0)}$, using the above inequality,
 we get
\begin{eqnarray}\label{4.11}
|u(x)|&=&\int_a^1\frac{\partial}{\partial t}|u(tx)|\mathrm{d}t \ \ =
\ \ \int_a^1\frac{1}{2|u(tx)|}\frac{\p}{\p t}|u(tx)|^2
\mathrm{d}t
\\ &=& \int_a^1\sum_i\frac{x_i}{2|u(tx)|} \frac{\p}{\partial
x_i}(\sum_\alpha (u^\alpha)^2)\mid_{tx}\mathrm{d}t \ \ \leqq \ \
\int_a^1\kappa|x|\mathrm{d}t \ \ \leqq \ \ \kappa.\nonumber
\end{eqnarray}
For any $x\in \mathbb{R}^n\backslash B_1(0)$, $u(x)\neq0$, there is
some  $b\geqq 1/|x|$, such that, for any $\tilde{b}\geqq 1/|x|$, and
$u(sx)\neq0$ holding for any $s\in(\tilde{b},1]$, we have $b\leqq\tilde{b}$. Then, by \eqref{4.10},
we have
\begin{eqnarray}\label{4.12}
\ \ \ \ \ \ |u(x)|&\leqq&
\left|\int_b^1\frac{d}{dt}|u(tx)|\mathrm{d}t\right|+\left|u(\frac{x}{|x|})\right|\
\
\leqq\ \ \int_b^1\frac{x_i}{2|u(tx)|}\frac{\p}{\partial x_i}(\sum_\alpha (u^\alpha)^2)\mid_{tx}\mathrm{d}t+\kappa\\
&\leqq&\ \ \int_b^{1}|x|\mathrm{d}t+\kappa\ \ \leqq \ \
|x|(1-b)+\kappa \ \ \leqq \ \ |x|-1+\kappa.\nonumber
\end{eqnarray}
Combining \eqref{4.11} and \eqref{4.12}, we get
\begin{equation}\aligned\label{4.13}
|u(x)|^2 \ \ \leqq \ \ |x|^2-2(1-\kappa)|x|+2, \ \ \ \mathrm{for\ all}\ x\in\R^n.
\endaligned
\end{equation}
Combining (\ref{4.8}) and (\ref{4.13}), we obtain
\begin{equation}\aligned\label{4.14}
\int_M\frac{|\nabla f|^2}{f}\eta^2\rho\ \ \leqq \ \
4\int_{\R^n}g^{ij}\eta_i\eta_j e^{1-(1-\kappa)|x|}dx\ \ \leqq \ \
4\int_{\R^n}\f{|\nabla\eta|^2}{\det(g)}e^{1-(1-\kappa)|x|}dx.
\endaligned
\end{equation}
Now we choose $\eta$.  For any positive $r$, let $\eta=\eta(|x|)$ be
a cut-off function in $\R$, $\eta\equiv1$ in $[0,r]$, $\eta\equiv0$
in $[2r,+\infty)$, and $|\eta'|\le C/r$, where $C$ is some constant
not depending on $r$. Since $\det(g)$ has subponential decay in
$|x|$ (c.f. Theorem \ref{thm2}), letting $r$ go to infinity in
\eqref{4.14}, we have $\nabla f=0$. Then $*dx$ is a constant. By
(\ref{4.0}) and (\ref{4.6}), we obtain $h_{ij}^{\alpha}=0$, which
implies the theorem. \qed

\bibliographystyle{amsplain}

\begin{thebibliography}{10}

\bibitem{AL} U. Abresch and J. Langer, {\it The normalized curve shortening
ow and homothetic solutions}, J. Diff. Geom. {\bf 23} (1986), no. 2,
175-196.



\bibitem{ACH2} A. Chau, J.Y. Chen, and W.Y. He,
{\it Entire self-similar solutions to Lagrangian mean curvature
flow},  arXiv: 0905.3869.

\bibitem{CCY} A. Chau, J.Y. Chen, and Y. Yuan,
{\it  Rigidity of entire self-shrinking solutions to curvature
flow,} arXiv:1003.3246v1.

\bibitem{CM1} T. H. Colding, and W. P. Minicozzi II, {\it Shapes of embedded
minimal surfaces,} Proc. Natl. Acad. Sci. USA 103 (2006), no. 30,
11106-11111.

\bibitem{CM} T.H. Colding, and W.P. Minicozzi II, {\it Generic mean curvature flow I;
generic singularities}, preprint, arxiv:0908.3788v1.

\bibitem{CM7} T.H. Colding, and W.P. Minicozzi II, {\it Smooth compactness of self-shrinkers,} preprint,
arXiv:0907.2594v1.

\bibitem{EH} K. Ecker and G. Huisken,
{\it Mean curvature evolution of entire graphs}, Annals of
Mathematics, {\bf 130} (1989), 453-471.

\bibitem{HL}  R. Harvey, H.B. Lawson, Jr., {\it Calibrated geometry}.
Acta. Math. 148(1982), 47-157.

\bibitem{HW} R. Huang, and  Z. Wang, {\it On the entire
self-shrinking solutions
 to Lagrangian mean curvature flow,} to appear in Calculus of
 variation and partial differential equations.

\bibitem{H1}  G. Huisken, {\it Flow by mean curvature of convex surfaces into sphere,} J.
Diff. Geom. 22 (1984), no. 1, 237-266.

\bibitem{H2} G. Huisken, {\it Asymptotic behavior for singularities of the mean
curvature flow,} J. Diff. Geom. 31 (1990), no. 1, 285-299.



\bibitem{H3} G. Huisken, {\it Local and global behaviour of hypersurfaces moving by
mean curvature,} Differential geometry: partial differential
equations on manifolds, 175-191, Proc. Sympos. Pure Math. 54 (1993),
Amer. Math. Soc.

\bibitem{JXY} J. Jost, Y.L. Xin, and L. Yang, {\it The Gauss image of entire graphs of high codimension and Benstein type theorems,}
arXiv:1009.3901v1.

\bibitem{KM} S. Kleene and N.M. M{\o}ller, {\it Self-shrinkers with a rotaion
symmetry,} arXiv:1008.1609v1.

\bibitem{Smo} K. Smozyk,
{\it Self-shrinkers of the mean curvature flow in arbitrary
codimension,} International Mathematics Research Notices, {\bf
2005(48)}, (2005) 2983-3004.

\bibitem{SWX} K. Smozyk, G. Wang, and Y.L. Xin, {\it Bernstein type theorems with flat
normal bundle,} Calculus of
 variation and partial differential equations, 2006, 26(1), 57-67.

\bibitem{Lwang} L. Wang,
{\it A Benstein type theorem for self-similar shrinkers,}
to appear in Geom. Dedicata, arXiv:0912:1809 v1.

\bibitem{Wang1} M-T. Wang,
{\it Interior gradient bounds for solutions to the minimal surfaces
syetem,} Amer. J. of math. {\bf(126)4}, (2004), 921-934.

\bibitem{Wang2} M-T. Wang, {\it Long-time existence and convergence of graphic mean curvature
flow in arbitrary codimension,} Invent. Math.  148 (2002) {\bf3},
525-543.


\bibitem{Xin} Y.L. Xin,
{\it Minimal submainfolds and related topics,} World Scientific
Publ., (2003).

\bibitem{Y}Y. Yuan,
{\it A Bernstein problem for special Lagrange equations,} Invent.
Math. {\bf 150}(2002), 117-125.
\end{thebibliography}

\end{document}